\documentclass[amssymb,verbatim,10pt]{amsart}
\usepackage[fontsize = 10bp]{fontsize}
\usepackage[font=footnotesize,labelfont=bf]{caption}
\usepackage{scalerel}
\usepackage{stackengine,wasysym}
\usepackage{tikz}        
\usepackage{float}

\usepackage{amssymb,latexsym, amsmath, amscd, array, graphicx,pb-diagram}
\usepackage{hyperref}
\usepackage{color}
\usepackage{mathtools}
\usepackage{extarrows}

\usepackage{tikz-cd}

\makeatletter
\def\@tocline#1#2#3#4#5#6#7{\relax
  \ifnum #1>\c@tocdepth % then omit
  \else
    \par \addpenalty\@secpenalty\addvspace{#2}%
    \begingroup \hyphenpenalty\@M
    \@ifempty{#4}{%
      \@tempdima\csname r@tocindent\number#1\endcsname\relax
    }{%
      \@tempdima#4\relax
    }%
    \parindent\z@ \leftskip#3\relax \advance\leftskip\@tempdima\relax
    \rightskip\@pnumwidth plus4em \parfillskip-\@pnumwidth
    #5\leavevmode\hskip-\@tempdima
      \ifcase #1
       \or\or \hskip 1em \or \hskip 2em \else \hskip 3em \fi%
      #6\nobreak\relax
    \hfill\hbox to\@pnumwidth{\@tocpagenum{#7}}\par% <---- \dotfill -> \hfill
    \nobreak
    \endgroup
  \fi}
\makeatother

\usepackage{color}
\hypersetup{
    colorlinks=true,
    linkcolor=red,
    urlcolor=black,
    citecolor=blue
           }
           
\usepackage{amsthm}
\usepackage{hyperref}
\usepackage[all]{xy}
\usepackage[scr]{rsfso}

\numberwithin{equation}{section}

\swapnumbers

\theoremstyle{plain}

\newtheorem{thm}{Theorem}[section]
\newtheorem{theorem}[thm]{Theorem}

\newtheorem{lemma}[thm]{Lemma}

\newtheorem{prop}[thm]{Proposition}
\newtheorem{cor}[thm]{Corollary}

\theoremstyle{definition}

\newtheorem{remark}[thm]{Remark}

  \DeclareMathOperator{\rank}{rank}

 \newcommand{\Wi}{\widetilde}

\DeclareMathOperator{\cd}{{\rm cd}}

\DeclareMathOperator{\cat}{{\mathsf{cat}}}

\DeclareMathOperator{\Ker}{{\rm Ker}}

\usepackage{enumerate}

\def\pr{\protect\operatorname{pr}}

\newcommand \pa[2]{\frac{\partial #1}{\partial #2}}

\def\C{{\mathbb C}}
\def\Z{{\mathbb Z}}

\def\R{{\R}}
\def\R{{\mathbb R}}

\def\1{\hbox{\rm\rlap {1}\hskip.03in{\rom I}}}
\def\Bbbone{{\rm1\mathchoice{\kern-0.25em}{\kern-0.25em}
{\kern-0.2em}{\kern-0.2em}I}}

\def\pa{\partial}

\def\wt{\widetilde}
\def\wh{\widehat}

\def\ov{\overline}

\usepackage{comment}

\long\def\forget#1\forgotten{} %

\newcommand\ver[1]{\marginpar{\tiny Changed in Ver \VER}}

\date{\today}

\begin{document}

\title[On totally $C$-symplectically aspherical manifolds]{On totally $c$-symplectically aspherical manifolds}

\author[L.~F.~Di~Cerbo]{Luca~F.~Di~Cerbo}

\author[A.~Dranishnikov]{Alexander~Dranishnikov} 

\author[E.~Jauhari]{Ekansh~Jauhari}

\address{Luca F. Di Cerbo, Department of Mathematics, University of Florida, 358 Little Hall, Gainesville, FL 32611-8105, USA.} 
\email{ldicerbo@ufl.edu}

\address{Alexander N. Dranishnikov, Department of Mathematics, University of Florida, 358 Little Hall, Gainesville, FL 32611-8105, USA.}
\email{dranish@ufl.edu}

\address{Ekansh Jauhari, Department of Mathematics, University of Florida, 358 Little Hall, Gainesville, FL 32611-8105, USA.}
\email{ekanshjauhari@ufl.edu}

\begin{abstract}
We construct examples of $c$-symplectically aspherical near-symplectic manifolds with non-trivial second homotopy group.
\end{abstract}



%\dedicatory{Dedicated to the memory of Yuli Rudyak}

\maketitle
%\tableofcontents

\section{Introduction}
Let $M$ be a closed smooth orientable $2n$-manifold for some $n\ge 1$. A closed differential $2$-form $\omega$ on $M$ is called \emph{aspherical} if
\[
\int_{S^2}f^*\omega=0
\]
for each smooth map $f\colon S^2\to M$. In other words, $\omega$ is an aspherical form if it evaluates trivially on each real spherical homology class of $M$, or equivalently, if the closed $2$-form $\pi^*\omega$ on $\wt M$, pulled back along the universal Riemannian covering $\pi\colon \wt M\to M$, is exact.

Aspherical symplectic forms were first introduced by Floer in his attempts to attack the Arnol'd conjecture~\cite{Floer}. Indeed, the latter was proven true for most \emph{symplectically aspherical manifolds} (i.e., closed manifolds equipped with aspherical symplectic forms), see~\cite{RO}. It is also well-known that Floer homology, the Lusternik--Schnirelmann theory for Lagrangian intersections, and the theory of $J$-holomorphic curves are simpler for symplectically aspherical manifolds, so naturally, the symplectic asphericity condition has appeared in a number of classical theorems in symplectic topology (see~\cite{KRT2} and the references therein).

Besides closed symplectic manifolds with vanishing second homotopy group that are symplectically aspherical for trivial reasons, there are many examples of \emph{non-trivial} symplectically aspherical manifolds (i.e., those with non-trivial $\pi_2$); see~\cite{Go2,IKRT,FMS,KRT1,DCDJ,DCDJ2} for some examples and their other interesting features. In this paper we study a related condition for \emph{$c$-symplectic} forms. Recall that a closed $2$-form on a closed smooth orientable $2n$-manifold is $c$-symplectic if its de Rham cohomology class has non-zero $n$-th power~\cite{LO,CO}. We call a manifold \emph{totally $c$-symplectically aspherical} if every $c$-symplectic form on it is aspherical. This condition is strongly constrained by the spherical homology of the manifold; in particular, it is automatic when every real spherical homology class in degree two vanishes.

Our goal is to construct explicit totally $c$-symplectically aspherical manifolds which nevertheless have non-trivial second homotopy group and carry near-symplectic structures. Near-symplectic structures, which are closed $2$-forms that are symplectic outside a codimension-$3$ submanifold, provide a natural setting for this problem, especially in dimension~$4$; see~\cite{Taubes2,Honda,GK,ADK,Vera}. While near-symplectic and $c$-symplectic structures are well studied, it seems that non-trivial aspherical $c$-symplectic structures on non-symplectic near-symplectic manifolds have received comparatively little attention.

Our main examples are a family of closed oriented smooth $4$-manifolds
\[
S_g:=SP^2(M_g)\ \#_{\Sigma}\ \ov{SP^2(M_g)},\qquad g\ge 3,
\]
constructed as generalized connected sums involving the symmetric squares of closed Riemann surfaces. The latter have been extensively studied in topology and geometry; see, for instance,~\cite{Mac,ACGH,BGZ,OS,DCDJ}. We prove in Theorem~\ref{thm: main} that each $S_g$ admits both $c$-symplectic and near-symplectic structures, that
\[
\pi_2(S_g)\otimes\R\ne 0,
\]
and that the real second Hurewicz map is trivial. Consequently every closed differential $2$-form on $S_g$ is aspherical, and in particular $S_g$ is a non-trivial totally $c$-symplectically aspherical near-symplectic $4$-manifold. The manifolds $S_g$ are never complex, and for even $g$ they do not admit almost complex, hence symplectic, structures.

The verification of these properties constitutes the bulk of the paper and uses several topological tools, including generalized connected sums, Whitney's embedding theorem, Mayer--Vietoris arguments, covering-space methods, and spectral sequences. The presence of non-trivial $\pi_2$ together with a trivial real Hurewicz map is the key mechanism behind total $c$-symplectic asphericity. Taking products with tori then yields examples in every even dimension at least four; see Corollary~\ref{cor: each even dim}.

Thus, the results of this paper show that total $c$-symplectic asphericity is compatible with non-trivial second homotopy and with near-symplectic geometry, while remaining sufficiently restrictive to impose strong conditions on spherical homology.

\subsection*{Arrangement of the paper}
We begin with a brief review of near-symplectic and aspherical $c$-symplectic forms in Section~\ref{sec: preliminaries}, where we also record facts about the symmetric squares of closed Riemann surfaces. In Section~\ref{subsec: construction} we construct the manifolds $S_g$, state the main theorem, and establish their $c$-symplectic structures together with several basic topological properties. Section~\ref{sec: homotopy of N} studies the homotopy and homology of the sphere complement used in the construction. In Section~\ref{sec: proof section} we compute the relevant (co)homology, analyze the real Hurewicz map and the second homotopy group, and complete the proof of Theorem~\ref{thm: main}; we also prove the higher-dimensional product consequence. Finally, Section~\ref{sec: exceptional} discusses the sharpness of the construction in dimension and genus.

\section{Background and preliminary facts}\label{sec: preliminaries}
This section gives a quick review of the main objects of this paper and records some facts that will be useful later. 

\subsection{Near-symplectic forms}\label{subsec: near-symplectic} 
Let $M$ be a closed smooth oriented $2n$-manifold for some $n\ge 2$. A closed differential $2$-form $\omega$ on $M$ is called \emph{near-symplectic} if it is symplectic outside a codimension-$3$ submanifold of $M$ where it vanishes transversely, i.e., for every $p\in M$, either $\omega^n_{p}>0$, or $\omega^{n-1}_p=0$ and $\omega^{n-2}_p\ne 0$ at a codimension-$3$ submanifold of $M$, see~\cite{ADK,Vera}.

In dimension $4$, a near-symplectic form on $M$ exists if and only if there is a de Rham cohomology class whose square induces the given orientation of $M$, see~\cite{Taubes2,Honda}. A near-symplectic form determines a \emph{singular Lefschetz pencil} having quadratic singularities along its $1$-dimensional zero set \cite{ADK,AK}; this is analogous to Donaldson's famous result in the symplectic setting~\cite{donaldson}. In turn, a near-symplectic structure is implied by the existence of certain singular Lefschetz pencils and degree $2$ (co)homology classes~\cite{ADK}; this too is analogous to Gompf's celebrated result in the symplectic setting~\cite{GS}. We refer to~\cite{GK,EF} for other noteworthy results concerning near-symplectic $4$-manifolds.

\subsection{Aspherical $c$-symplectic forms}\label{subsec: cohomologically symplectic} Let $M$ be a closed smooth orientable $2n$-manifold for some $n\ge 1$. A closed differential $2$-form $\omega$ on $M$ is called \emph{$c$-symplectic} (short for cohomologically symplectic) if its de Rham cohomology class $[\omega]\in H^2(M;\R)$ satisfies $[\omega]^n\ne 0$ (see, for instance,~\cite{LO}). These forms mimic the cohomological property of symplectic forms, and as such are independent of the existence of almost complex structures (consider, for example, $\C P^2\#\C P^2$).
%has a $c$-symplectic form, but no almost complex structure compatible with its orientation (see~\cite[Proposition~1.3.1]{Aud} for a generalization).

The existence of aspherical $c$-symplectic forms on a smooth manifold severely constrains its topology and geometry. For instance, a $c$-symplectically aspherical manifold $M$ is \emph{essential} (in the sense of Gromov~\cite{Gr-ess}), $S^1$-actions on $M$ are not Hamiltonian~\cite{LO}, and $\pi_1(M)$ is always infinite. In fact, one has 
\begin{equation}\label{eq: cd}
\cd_\R(\pi_1(M))\ge \dim_\R(M),
\end{equation}
for the real cohomological dimension, see~\cite[Lemma~4.2]{LO}. We refer to~\cite{LO,RO,IKRT,KRT2,Borat} for other such constraints and interesting results.

\subsection{Symmetric squares of surfaces}\label{sec: SP}
%As advertised earlier, the building blocks of our totally $c$-symplectically examples in dimension $4$ are the symmetric squares of closed Riemann surfaces. 
For a surface $M_g$ of genus $g\ge 0$, its \emph{symmetric square} $SP^2(M_g)$ is the orbit space of the natural $\Z_2$-action on $M_g\times M_g$. It is a closed smooth orientable $4$-manifold. Indeed, $SP^2(M_0)\cong \C P^2$, $SP^2(M_1)$ is a $\C P^1$-bundle over $T^2$, and $SP^2(M_2)\cong T^4\# \ov{\C P^2}$. The Jacobian of $SP^2(M_g)$ coincides with that of $M_g$, so the Albanese map is precisely the Abel--Jacobi map 
\[
\mu_2\colon SP^2(M_g)\to T^{2g}:=\C^g/\Lambda, \quad \Lambda\cong\Z^{2g},
\]
see~\cite{ACGH}. We refer to~\cite{BGZ} (and the references therein) for various other interesting facts about the more general symmetric \emph{products} of surfaces.

For future use, we recall from~\cite{Mac} that if $\{a_i,b_i\mid1\le i \le g\}$ is a symplectic basis of $H_1(T^{2g};\Z)$, $c$ is the fundamental class of $M_g$, and we use the same notations for the images of these classes in $SP^2(M_g)$ under the appropriate maps, then the cohomology ring $H^*(SP^2(M_g);\Z)$ is torsion-free and is generated by the classes $\{a_i^*,b_i^*,c^*\mid 1\le i \le g\}$ (where $\ast$ in the superscript denotes the respective Hom duals in $SP^2(M_g)$), subject only to the following relation:
\[
a_{i_1}^*\cdots a_{i_\ell}^*b_{j_1}^*\cdots b_{j_m}^*(c^*-a_{k_1}^*b_{k_1}^*)\cdots(c^*-a_{k_r}^*b_{k_r}^*)(c^*)^s=0
\]
if $\ell+m+2r+s>2$ for any \emph{distinct} set of indices $i_1,\dots,i_\ell$, $j_1,\dots,j_m$, and $k_1,\dots, k_r$. It is also well-known (see, for example,~\cite{Ka2,OS,DCDJ}) that  
\[
\pi_1(SP^2(M_g))\cong \wt H_1(M_g;\Z)\cong \Z^{2g},
\]
$\pi_2(SP^2(M_g))$ is a $\Z^{2g}$-module, and there exists a generator 
\begin{equation}\label{eq: spherical homo}
u=c-\sum_{i=1}^ga_i\cdot b_i  \ \in H_2(SP^2(M_g);\R)   
\end{equation}
of real spherical homology classes of $SP^2(M_g)$. Here, $\cdot$ is the Pontryagin product. Moreover, $H_2(SP^2(M_g);\R)$ is generated by $c$ and the Pontryagin products $a_i\cdot a_j$ and $b_i \cdot b_j$ for $i < j$, and $a_i \cdot b_j$ for all $i, j$. Finally, we mention that the Poincar\'e duals $PD$ of these Pontryagin products were computed in~\cite{Ka2}. Specifically, $PD(a_i\cdot b_i)=c^*-a_i^*b_i^*=c^*+b_i^*a_i^*$; see the proof of Theorem~7.1 in~\cite{DCDJ2}.

\section{Construction of the examples}\label{subsec: construction}

In this section, we give examples of totally $c$-symplectically aspherical near-symplectic $4$-manifolds with non-trivial second homotopy group. As mentioned earlier, the verification of the asserted properties of these examples will continue in Sections~\ref{sec: homotopy of N} and~\ref{sec: proof section} as well.

%\subsection{Description of the examples}\label{subsec: describe} 
The building blocks of our $4$-dimensional examples are the symmetric products of closed Riemann surfaces. Fix some $g\ge 0$ and consider the closed smooth oriented $4$-manifold $SP^2(M_g)$; the symmetric square of a Riemann surface $M_g$ of genus $g$. There is a non-trivial spherical homology class $u\in H_2(SP^2(M_g);\R)$, see~\eqref{eq: spherical homo}. Let 
\[
\Sigma\subset SP^2(M_g)
\]
be a smoothly embedded $2$-sphere that represents $u$, and let $W\subset SP^2(M_g)$ be its tubular neighborhood that is canonically identified with the total space of the normal bundle of the smooth embedding $\Sigma\hookrightarrow SP^2(M_g)$. Define 
\[
N=SP^2(M_g)\setminus\text{Int}(W),
\]
so that the common boundary of $N$ and $W$ is a closed oriented $3$-manifold $L$.
%$L:=L_m$, which is a 3-dimensional lens space fibered over $\Sigma$ with fiber a circle $\gamma$ that represents a generator of $\pi_1(L)\cong\Z_m$. 
Take $V=\ov N$ to be the homeomorphic copy of $N$ with the opposite (or reversed) orientation, so that $\pa V=\ov{L}$. The identity map $\phi\colon L\to \ov L$ is an orientation-reversing diffeomorphism. Finally, glue $N$ and $\ov N$ along $L$ using $\phi$ to form the closed smooth oriented $4$-manifold $N\cup_\phi\ov N$, which we denote by
\[
S_g:=SP^2(M_g)\ \#_{\Sigma}\ \ov{SP^2(M_g)}.
\]
Here, $\ov{SP^2(M_g)}$ is the copy of the oriented $SP^2(M_g)$ with the opposite orientation. Moving forward, we will primarily follow the notations introduced here.

We now state the main result of this paper.

\begin{theorem}\label{thm: main}
For each $g\ge 3$, the closed oriented smooth $4$-manifold $S_g$ admits $c$-symplectic and near-symplectic structures, and satisfies
\[
\pi_2(S_g)\otimes\R\ne 0 \quad \text{and} \quad \textup{Im}(h_2\otimes 1_\R\colon \pi_2(S_g)\otimes\R\to H_2(S_g;\R))=0,
\]
where $h_2$ is the second Hurewicz map. In particular, $S_g$ for $g\ge 3$ is a non-trivial totally $c$-symplectically aspherical near-symplectic $4$-manifold.
\end{theorem}

For $g$ even, $S_g$ cannot support almost complex structures (so $S_g$ is not symplectic for $g$ even), see Proposition~\ref{prop: no almost cplx} below. More generally, $S_g$ cannot support complex structures for any $g$ (so they are never K\"ahler), see Subsection~\ref{sec: cohomology Xg}. We also note that the genus restriction in Theorem~\ref{thm: main} cannot be dropped (\emph{cf}. Section~\ref{sec: exceptional}).

Taking products of $S_g$ with tori yields the following (see Subsection~\ref{sec: real Hurewicz}).

\begin{cor}\label{cor: each even dim}
   In each even dimension $\ge 4$, there exists an infinite family of totally $c$-symplectically aspherical near-symplectic manifolds that are not K\"ahler.
\end{cor}

\subsection{$c$-symplectic structures on $S_g$}

While a near-symplectic form on a closed $4$-manifold induces a $c$-symplectic structure on it (\emph{cf.} Section~\ref{sec: preliminaries}), we can directly show explicit $c$-symplectic structures on $S_g$ for $g\ge 3$. The idea is to realize $S_g$ as smooth submanifolds of higher dimensional tori.

The following lemma will be crucial for computational purposes.

\begin{lemma}\label{lem: pd u square}
    For $g\ge 0$, the self-intersection of $\Sigma$ in $SP^2(M_g)$ is equal to $1-g$.
\end{lemma}

\begin{proof}
    Since $\Sigma$ represents a real spherical homology class $u$, we must prove that $(PD(u))^2=1-g$.  Recall from Subsection~\ref{sec: SP} that $u=c-\sum_{i=1}^ga_i\cdot b_i$, $PD(c)=c^*$, and $PD(a_i\cdot b_i)=c^*-a_i^*b_i^*$. So, we get
    \[
    (PD(u))^2=(c^*)^2-2\sum_{i=1}^gc^*(c^*-a_i^*b_i^*)+\left(\sum_{i=1}^g(c^*-a_i^*b_i^*)\right)^2.
    \]
    The first term above is $1$ and the second is $0$ by Macdonald's relations~\cite{Mac}. Similarly, 
    \[
    (c^*-a_i^*b_i^*)^2=-a_i^*b_i^*(c^*-a_i^*b_i^*)=-a_i^*b_i^*c^*=c^*(c^*-a_i^*b_i^*)-1=-1
    \]
    because $a_i^*b_j^*=-b_j^*a_i^*$ and $(a_i^*)^2=0=(b_j^*)^2$ for all $i,j$. Also, for $j\ne i$, we have $(c^*-a_i^*b_i^*)(c^*-a_j^*b_j^*)=-a_i^*b_i^*(c^*-a_j^*b_j^*)=0$ by Macdonald's relations. Hence, the third term above (the squared term) is $-g$. Therefore, $(PD(u))^2=1-g$.
\end{proof}

Consider the abelian variety $T^{2g+4}\cong \C^{g+2}/\Lambda$ (where $\Lambda\cong\Z^{2g+4}$ is a lattice), the standard K\"ahler form $\omega$ on $T^{2g+4}$ induced using the canonical $\Lambda$-invariant Euclidean K\"ahler form on $\C^{g+2}$, the closed surface $M_g$ of genus $g\ge 3$ equipped with a complex structure \emph{generic} in its moduli space, and the smooth projective variety $SP^2(M_g)$ that parametrizes effective divisors of degree $2$ on $M_g$. It follows from the proof of Theorem~4.1 in~\cite{DCDJ} that the composition
\[
\mu\colon SP^2(M_g)\xlongrightarrow{\mu_2}T^{2g}\xhookrightarrow{i}T^{2g+4}
\]
is a holomorphic embedding, where $\mu_2$ is the Abel--Jacobi map. Furthermore, we have the following. 

\begin{prop}\label{prop: tori embedding}
    For $g\ge 3$, there exists a smooth embedding $\nu\colon S_g\to T^{2g+4}$ such that the pullback form $\nu^*\omega$ on $S_g$ is $c$-symplectic.
\end{prop}

\begin{proof}
Suppose $\mu_{|N}\colon N\to T^{2g+4}$ is the restriction of $\mu\colon SP^2(M_g)\to T^{2g+4}$ to $N$. Note that $S_g=N \cup_{\phi}\ov N$ and 
\[
\pa N\cong L_m=:L
\]
is a closed smooth $3$-dimensional lens space for $m=g-1\ge 2$ (\emph{cf}. Lemma~\ref{lem: pd u square}), which is fibered over $\Sigma$ with fiber a circle $\gamma$ that represents a generator of the finite cyclic group $\pi_1(L)\cong\Z_m$, see~\eqref{eq: bundle} below. Due to fundamental group reasons, the restriction $\mu_{|L}\colon L\to Y$ of $\mu_{|N}$ is null-homotopic, so $\mu_{|N}$ extends to a map $\nu'\colon S_g\to T^{2g+4}$, which can be approximated by a \emph{homotopic} smooth embedding $\nu\colon S_g\to T^{2g+4}$ in view of Whitney's embedding theorems. Now, $\nu$ factors through $S_g/L$, so the cohomology class $[\nu^*\omega]$ comes from a class $a\in H^2(S_g/L)$.
%, i.e., $[\nu^*\omega]=q^*(a)$, where $q\colon S_g\to S_g/L$ is the quotient map. 
We note that 
\[
S_g/L\cong(SP^2(M_g)/\Sigma)\vee(SP^2(M_g)/\Sigma),
\]
so
by our construction, 
\[
a=x\oplus 0\in H^2(SP^2(M_g)/\Sigma)\oplus H^2(SP^2(M_g)/\Sigma)\cong H^2(S_g/L),
\]
where $x$ maps to $[\mu^*\omega]$. In particular, $x^2\ne 0$, which implies $a^2\ne 0$.
In the exact sequence
\[
\cdots\to\Z\cong H^3(L)\xlongrightarrow{j} H^4(S_g/L)\cong \Z\oplus\Z\longrightarrow H^4(S_g)\cong\Z\to 0,
\]
the class $a^2=x^2\oplus 0$ is not in the image of $j$ since $j(y)=\pm j_N(y)\oplus j_N(y)$, where $j_N$ is the homomorphism induced in cohomology by the inclusion $L\hookrightarrow N$. By exactness, it follows that $[\nu^*\omega]^2\ne 0$.
\end{proof}

\subsection{Viewing $S_g$ as generalized connected sums}
Up to homeomorphism, the closed $4$-manifold $S_g$ is the \emph{generalized connected sum} of $M_1:=SP^2(M_g)$ and $M_2:=\ov{SP^2(M_g)}$ along $\Sigma$. To see this, we begin by letting $j_i\colon \Sigma\to M_i$ denote the smooth embedding for $i\in\{1,2\}$, and we let $\eta_i$ be the normal bundle of $j_i$ that is canonically identified with a tubular neighborhood $W_i\subset M_i$ of $j_i(\Sigma)$. The Euler class of $\eta_i$, being equal to the self-intersection number of $\Sigma$ in $M_i$, is $(-1)^i(g-1)$ for $i\in\{1,2\}$ by Lemma~\ref{lem: pd u square}. The Euler classes have opposite signs for $g\ne 1$, so
%$W_1$ and $W_2$ are diffeomorphic and 
there exists a fiberwise orientation-reversing diffeomorphism $\eta_1\to \eta_2$ in this case, and thus an orientation-preserving diffeomorphism 
\[
\psi\colon W_1\setminus j_1(\Sigma)\to W_2\setminus j_2(\Sigma),
\]
see, for example,~\cite[Page~534]{Go1}. Then $S_g$ is obtained by gluing $W_1\setminus j_1(\Sigma)$ to $W_2\setminus j_2(\Sigma)$ using $\psi$ in the disjoint union $(M_1\setminus j_1(\Sigma))\sqcup (M_2\setminus j_2(\Sigma))$.  

In the symplectic setting, Gompf's result~\cite{Go1} says, in particular, that if $M_i$ and $j_i$ are symplectic for $i\in\{1,2\}$, then $M_1\ \#_{\Sigma}\ M_2$ is symplectic as well. However, that result does not apply here since the manifold $M_2=\ov{SP^2(M_g)}$ is not symplectic for any $g\ne 1$, as we explain below in Proposition~\ref{prop: opposite orientation}. First, we need the following elementary lemma.

\begin{lemma}\label{lem: euler and sig}
The Euler characteristic and signature of $SP^2(M_g)$, respectively, are 
\[
\chi(SP^2(M_g))=2g^2-5g+3\quad \text{and}\quad \sigma(SP^2(M_g))=1-g.
\]   
In particular, $b_2^+(SP^2(M_g))=g^2-g+1$ and $b_2^-(SP^2(M_g))=g^2$.
\end{lemma}
\begin{proof}
Poincar\'e duality and the isomorphism $\pi_1(SP^2(M_g))\cong\Z^{2g}$ implies that $b_0(SP^2(M_g))=1=b_4(SP^2(M_g))$ and $b_1(SP^2(M_g))=2g=b_3(SP^2(M_g))$. By the description of the basis of $H_2(SP^2(M_g);\R)$ from Subsection~\ref{sec: SP}, we also have 
\[
b_2(SP^2(M_g))=2g^2-g+1,
\]
which implies that $\chi(SP^2(M_g))=2g^2-5g+3$.

For the signature, we proceed as follows. Since $(PD(c))^2=(c^*)^2=1$, the self-intersection of $c$ is $1$. It follows from the expressions for the Poincar\'e duals from~\cite{Ka2} that the self-intersection of $a_i\cdot b_i$ is $-1$ for $1\le i \le g$, and that of any other Pontryagin product in the basis of $H_2(SP^2(M_g);\R)$ is $0$. So, $g$ entries in the intersection matrix are $-1$, one is $1$, and others are $0$. Hence, $\sigma(SP^2(M_g))=1-g$.

The assertions for $b_2^+(SP^2(M_g))$ and $b_2^-(SP^2(M_g))$ are now easily verified.
\end{proof}

\begin{prop}\label{prop: opposite orientation}
Consider the closed oriented symplectic $4$-manifold $SP^2(M_g)$. For $g\ne 1$, the smooth $4$-manifold $\ov{SP^2(M_g)}$ cannot support a symplectic structure compatible with the opposite orientation. In fact, for $g$ even, $\ov{SP^2(M_g)}$ cannot support an almost complex structure compatible with the opposite orientation.
\end{prop}

\begin{proof}
We first look at the case of even $g\ge 0$. We have from Lemma~\ref{lem: euler and sig} that $\sigma(\ov{SP^2(M_g)})=-\sigma(SP^2(M_g))=g-1$ and $\chi(\ov{SP^2(M_g)})=2g^2-5g+3$. Clearly, $\chi(\ov{SP^2(M_g)})+\sigma(\ov{SP^2(M_g)})=2g^2-4g+2$, which is not divisible by $4$ for $g$ even. However, it is well known due to Hirzebruch and Milnor (see~\cite{BH,Milnor}) that the Todd genus 
\begin{equation}\label{eq: todd}
\frac{\chi(Y)+\sigma(Y)}{4}  
\end{equation}
of a closed $4$-manifold $Y$ is an integer if $Y$ admits an almost complex structure.

This technique does not give obstructions to symplectic structure in the case of odd $g$, so we proceed as follows for \emph{each} $g\ge 3$. The compact complex surface $SP^2(M_g)$ is of general type in this range (see, for instance,~\cite[Page~39]{Abr}). Suppose $\ov{SP^2(M_g)}$ admits an orientation-compatible symplectic structure. Because 
\[
b_2^+(\ov{SP^2(M_g)})=b_2^-(SP^2(M_g))=g^2>1
\]
in view of Lemma~\ref{lem: euler and sig}, the Seiberg--Witten invariant of the canonical class of $\ov{SP^2(M_g)}$ is equal to $\pm 1$ due to~\cite{Tau}. Then it follows from~\cite[Theorem~1]{Kot} that $\sigma(SP^2(M_g))\ge 0$. But this contradicts Lemma~\ref{lem: euler and sig}. 
\end{proof}

\begin{remark}\label{rem: dragichi result}
    The above proof does not work in the case $g=1$. In this case, the Abel--Jacobi map $\mu_2\colon SP^2(M_1)\to T^{2}$ is diffeomorphic to a smooth locally trivial $S^2$-bundle, where $[S^2]\ne 0$ in $H_2(SP^2(M_1);\R)$. As explained in~\cite[Page~85]{Dra}, Thurston's well-known construction~\cite{Thu} of symplectic forms on such $4$-manifolds produces respective orientation-compatible symplectic structures on both $SP^2(M_1)$ and $\ov{SP^2(M_1)}$.
\end{remark}

The next proposition justifies the study of $c$-symplectic and near-symplectic structures on $S_g$.

\begin{prop}\label{prop: no almost cplx}
    For $g\ge 0$ even, the closed $4$-manifold $S_g$ cannot support an almost complex structure, and hence a symplectic structure.
\end{prop}

\begin{proof}
First, note that it follows from the discussion so far that, up to homeomorphism, $S_g$ is oriented-cobordant to the disjoint union $SP^2(M_g)\sqcup \ov{SP^2(M_g)}$ (see~\cite[Page~535]{Go1}). Therefore, in particular, we have
\begin{equation}\label{eq: euler xg}
\chi(S_g)=\chi(SP^2(M_g))+\chi(\ov{SP^2(M_g)})-2\chi(\Sigma)=4g^2-10g+2
\end{equation}
in view of Lemma~\ref{lem: euler and sig}. Similarly, the signature of $S_g$ is given by
\begin{equation}\label{eq: sign xg}
\sigma(S_g)=\sigma(SP^2(M_g))+\sigma(\ov{SP^2(M_g)})=\sigma(SP^2(M_g))-\sigma(SP^2(M_g))=0.    
\end{equation}
Then, as before, we can use the Hirzebruch--Milnor obstruction to almost complex structures on closed $4$-manifolds~\cite{BH,Milnor}, that the Todd genus~\eqref{eq: todd} be an integer, to conclude the proof. Indeed, $\chi(S_g)+\sigma(S_g)=4g^2-10g+2$ is \emph{not} divisible by $4$ for any even $g$.
\end{proof}

\section{Homotopy and homology of the sphere complement}\label{sec: homotopy of N} 
In this section, we present several important  topological computations for the compact $4$-manifold $N:=SP^2(M_g)\setminus \text{Int}(W)$, which is the $\Sigma$-sphere complement in $SP^2(M_g)$ (\emph{cf.} Section~\ref{subsec: construction}), and some of its coverings that will be very useful later.

\subsection{The fundamental group of $N$}
Recall from Lemma~\ref{lem: pd u square} that the self-intersection of the smooth $2$-sphere $\Sigma$ that represents the spherical homology class $u$ in $M:=SP^2(M_g)$ is $\langle PD(u),u\rangle=1-g$ for each $g\ge 0$. Let $L=\partial N$.

%Recall that for $g\ge 3$, the boundary of the compact $4$-manifold $N=M\setminus\text{Int}(W)$ is a lens space fibered over $\Sigma$ with the fiber a circle $\gamma$ that represents a generator of $\pi_1(L)\cong\Z_m$, where $m=g-1\ge 2$.

\begin{prop}\label{prop: injective}
For $g\ge 0$, the map $\pi_1(L)\to\pi_1(N)$ induced by the inclusion $L\hookrightarrow N$ is injective.
\end{prop}
\begin{proof}
Note that $L$ is the total space of the spherical fiber bundle for the normal bundle of $\Sigma$:
\begin{equation}\label{eq: bundle}
   \gamma\cong  S^1\to L\to \Sigma.
\end{equation}
Here, the circle $\gamma$ is a generator of $\pi_1(L)\cong \Z/\langle g-1\rangle\Z$. If $\pi_1(L)\to \pi_1(N)$ were not injective, the self-intersection number of $\Sigma$ in $M$ would have to be less than $1-g$ in view of~\eqref{eq: bundle}. But the latter is not possible (\emph{cf}. Lemma~\ref{lem: pd u square}).
\end{proof}

\begin{lemma}\label{lem: intersections}
For $g\ge 1$, the intersection number of $\Sigma$ with the $2$-torus in $M$ determined by the Pontryagin product $a_i\cdot b_i\in H_2(M;\Z)$ is $1$ for $1\le i \le g$.
\end{lemma}

\begin{proof}
Recall from Subsection~\ref{sec: SP} that $PD(a_i\cdot b_i)=b_i^*a_i^*+c^*$ for $1 \le i \le g$, and so, $PD(u)=(1-g)c^*-\sum_{j=1}^g b_j^*a_j^*$. Then we have for $\langle PD(a_i\cdot b_i),u\rangle$ that
\[
PD(a_i\cdot b_i)\smile PD(u)= {\left(1-g\right)c^*b_i^*a_i^*}+{\left(1-g\right)(c^*)^2}-(c^*-a_i^*b_i^*)\sum_{j=1}^gb_j^*a_j^*,
\]
where we use the basic identities $(c^*)^2=1$ and $a_j^*b_j^*=-b_j^*a_j^*$ for each $j$. By what we said in the proof in Lemma~\ref{lem: pd u square}, the third term above simplifies to $-c^*b_i^*a_i^*$. Therefore, Macdonald's relations mentioned in Subsection~\ref{sec: SP} give
\[
PD(a_i\cdot b_i)\smile PD(u)= -gc^*b_i^*a_i^*+(1-g)(c^*)^2=1-gc^*(c^*-a_i^*b_i^*)=1.
\]
\begin{comment}
\[
=\left(1-g\right)c^*b_i^*a_i^*+(1-g)(c^*)^2-
\]
\[
= \left(1-g\right)c^*b_i^*a_i^*+{1-g}+b_i^*a_i^*\sum_{j\ne i}a_j^*b_j^*-{c^*\left(\sum_{j=1}^g\left(c^*-a_j^*b_j^*\right)-gc^*\right)}
\]
\[
= {\left(1-g\right)c^*b_i^*a_i^*}+{1-g}+b_i^*a_i^*\left(\sum_{j\ne i}\left(a_j^*b_j^*-c^*\right)\right)+(g-1)b_i^*a_i^*c^*+{g}
\]
\[
=1-\sum_{j\ne i}\left(a_i^*b_i^*\left(a_j^*b_j^*-c^*\right)\right)=1+0=1,
\]
\end{comment}
\end{proof}

Next, we describe the fundamental group of the sphere complement. 

\begin{prop}\label{prop: central extension}
For $g\ge 1$, there is a short exact sequence
\[
1\to\pi_1(L)\to\pi_1(N)\to \mathbb Z^{2g}\to 1,
\]
which is a central extension.
\end{prop}
\begin{proof}
The epimorphism in the sequence is defined by the inclusion $N\hookrightarrow M$, and the monomorphism is explained by Proposition~\ref{prop: injective}. By exactness, the kernel of $\pi_1(N)\to \Z^{2g}$ is normally generated by $\gamma$. We show that $\pi_1(L)=\langle\gamma\rangle$ is central.

We begin by showing that $\gamma=[a_i,b_i]$ in $\pi_1(N)$ for any $i\in\{1,\ldots g\}$; here we use the same notation $\gamma$ for its homotopy class $[\gamma]$ in $\pi_1(L)\subset \pi_1(N)$. The Pontryagin product $a_i\cdot b_i\in H_2(M;\Z)$ defines a singular $2$-torus $T_i$ in $M$ whose intersection number with $\Sigma$ is $1$ by Lemma~\ref{lem: intersections}. Because the fundamental group $\pi_1(M)\cong\Z^{2g}$ is \emph{good}, in the sense of Freedman and Quinn~\cite{FQ}, by performing the topological Whitney trick on $T_i$, we may assume that $T_i$ intersects $\Sigma$ at a single point. Thus, $\gamma$ is the boundary of a disk removed from the torus $T_i$, from which it follows that $\gamma$ is homotopic to the commutator $[a_i,b_i] =a_i^{-1}b_i^{-1}a_ib_i$. Now, since the intersection number of $-a_i\cdot b_i$ with $\Sigma$ is equal to $-1$ (\emph{cf}. Lemma~\ref{lem: intersections}) and since $b_i\cdot a_i=-a_i\cdot b_i$ (see, for instance,~\cite[Proposition~13.37]{Sw}), we get $b_i\cdot a_i^{-1}=-b_i\cdot a_i=a_i\cdot b_i$. Hence, $[b_i,a_i^{-1}]=\gamma$ as above. Then 
\[
a_i\gamma a_i^{-1}=a_i[a_i,b_i]a_i^{-1}=a_ia_i^{-1}b_i^{-1}a_ib_ia_i^{-1}=[b_i,a_i^{-1}]=\gamma.
\]
Since $a_i^{-1}\cdot b_i^{-1}=a_i\cdot b_i$,  we have that $[a_i^{-1},b_i^{-1}]=\gamma$. Then similarly we obtain
\[
b_i\gamma b_i^{-1}=b_i[b_i,a_i^{-1}]b_i^{-1}=b_ib_i^{-1}a_ib_ia_i^{-1}b_i^{-1}=[a_i^{-1},b_i^{-1}]=\gamma.
\]
Therefore, $\gamma$ commutes with $a_i$ and $b_i$ for each $i$.
\end{proof}

\subsection{The universal cover of $N$}\label{subsec: homology/homotopy}
Let $p:\widetilde M\to M$ be the universal covering of $M:=SP^2(M_g)$ for any fixed $g\ge 3$. We write $\wh N=p^{-1}(N)$, $\wh W=p^{-1}(W)$, and $\wh L=p^{-1}(L)$. Since the universal cover $\wt N$ of $N$ is the building block of the universal cover of $S_g$, we study the homology of $N$, $\wh N$, and $\wt N$. 

Moving forward, we adopt the convention of omitting coefficients in the case of \emph{integral} (co)homology, i.e., we write $H_*(-):=H_*(-;\Z)$ and $H^*(-):=H^*(-;\Z)$.

\begin{prop}\label{prop: about H_2}
We have that $H_2(\wh N)=H_2(\wh N,\wh L)=0$. 
\end{prop}
\begin{proof}
Consider the following commutative diagram, where maps in homology are generated by the inclusion $(\wh N,\wh L)\to (\wt M,\wh W)$:
\[
\begin{tikzcd}
    H_2(\wh L) \arrow{r}\arrow{d}
    &
    H_2(\wh N) \arrow{r}\arrow{d}
    &
    H_2(\wh N,\wh L) \arrow{d}{\cong}
    &
    \\
    H_2(\wh W) \arrow{r}{j}
    &
    H_2(\wt M) \arrow{r}
    &
    H_2(\wt M,\wh W)\arrow{r}
    &
    H_1(\wh W)= 0.
\end{tikzcd}
\]
In this diagram, the right vertical arrow is the excision isomorphism. The map $\pi_2(\wh W)\to \pi_2(\wt M)$ induced by the inclusion $\wh W\to \wt M$ is an isomorphism, hence so is the map $j$. Therefore, $H_2(\wt M,\wh W) =0$, and so $H_2(\wh N,\wh L)=0$.  Since $H_2(L)=0$ and $\wh L=\bigsqcup L$, we obtain $H_2(\wh L)=0$. It follows that $H_2(\wh N)=0$.
\end{proof}

\begin{prop}\label{prop: dim h_3(N)}
We have that $H_3(N)\cong\Z^{2g}$.
\end{prop}
\begin{proof} 
Note that excision and Lefschetz duality give isomorphisms 
\begin{equation*}\label{eq: excision iso}
H_k(M,N)\cong H^{4-k}(\Sigma) \ \text{ and } \ H^k(M,N)\cong H_{4-k}(\Sigma)
\end{equation*}
for $0\le k\le 4$, so we have $H_3(M,N)\cong 0$ and $H_4(M,N)\cong \mathbb Z$. Since $H_4(N)\cong 0$, the homology long exact sequence for the pair $(M,N)$ gives 
\[
0\to\Z\cong H_4(M)\xlongrightarrow{j}\Z\cong H_4(M,N)\to H_3(N)\to H_3(M)\to 0.
\]
Here, $j$ is an isomorphism. Therefore, $H_3(N)\cong H_3(M)$. Since $M:=SP^2(M_g)$, Poincar\'e duality gives $H_3(M)\cong H^1(M)\cong \Z^{2g}$. 
\end{proof}

Next, we proceed to present a lower bound to $\rank(\pi_2(N))$.  

\begin{prop}\label{prop: dim h_3(G)}
For $G=\pi_1(N)$, we have that
\[
\dim H_3(G,\R)={2g \choose 3}.
\]
\end{prop}
\begin{proof}
Consider the short exact sequence of Proposition~\ref{prop: central extension} for $g\ge 3$,
\[
1\to\Z_m\to G\to \Z^{2g}\to 1,
\]
where $m=g-1\ge 2$. By the Hochschild--Serre spectral sequence, and the fact that $H_q(\Z_m;\R)=0$ for $q>0$, we obtain
\[
\dim H_3(G,\R)=\dim H_3(\Z^{2g},\R)=\binom{2g}{3}.
\]
\end{proof}
\begin{prop}\label{prop: dim h_2(wt N)}
For the universal cover $\wt N$ of $N$, we have that 
\[
\dim H_2(\Wi N;\R)\ge {2g\choose 3}-2g.
\]
\end{prop}
\begin{proof}
Set $G=\pi_1(N)$. We consider the real homology spectral sequence (with local coefficients) for the Borel fibration 
\[
\wt N\to \wt N\times_GEG\to BG,
\]
see~\cite{Sp}. The total space of this fibration is homotopy equivalent to $N$, so
\begin{equation}\label{eq: dim sum}
    \dim H_3(N;\R)=\dim E_{3,0}^{\infty}+\dim E^{\infty}_{2,1}+\dim E^{\infty}_{1,2}+\dim E^{\infty}_{0,3},
\end{equation}
where all quantities are non-negative. Since $\wt N$ is simply connected, $E^2_{p,1}=0$ for all $p$. Thus, $E^r_{p,1}=E^2_{p,1}=0$ for each $r\ge 2$ and $p$. In particular, $E^{\infty}_{2,1}=0$. Similarly, we have $E^3_{p,0}=E^2_{p,0}$ for all $p$, which entails $E^3_{3,0}\cong H_3(G;\R)$. Because $E^r_{p,q}=0$ for each $r$ and $q$ if $p<0$, we get $E^r_{0,2}=E^2_{0,2}$ for all $r\ge 2$. In particular, we have $E^{3}_{0,2}\cong H_2(\wt N;\R)_G$, where the latter is the group of coinvariants. The first quadrant of the $E^3$ page looks as follows.
\begin{figure}[H]
    \centering
    \includegraphics[width=0.8\linewidth]{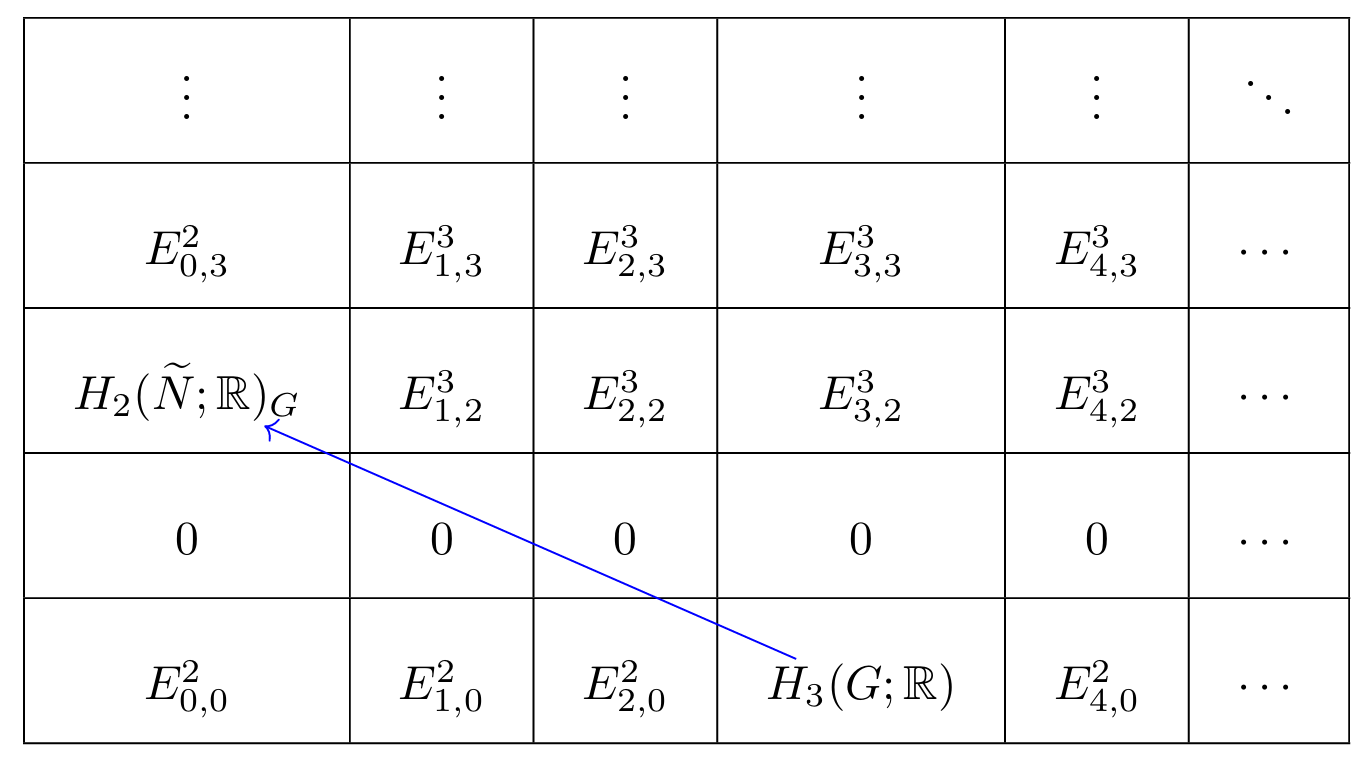}
%\caption{The first quadrant of the $E^3$ page of the spectral sequence.}
    \label{fig: spec seq}
\end{figure}
Here, the arrow in blue is the differential $d_3\colon E^3_{3,0}\to E^3_{0,2}$. Since $E^r_{p,q}=0$ for each $r$ and $q$ if $p<0$, we have $E^r_{3,0}=E^4_{3,0}$ for each $r\ge 4$, so that
\begin{equation}\label{eq: e infty page}
E^{\infty}_{3,0}=E^4_{3,0}=\Ker\left(d_3\colon E^3_{3,0}\to E^3_{0,2}\right).    
\end{equation}
From~\eqref{eq: dim sum} and~\eqref{eq: e infty page}, we get $\dim H_3(N;\R)\ge \dim \Ker(d_3)\ge \dim E^3_{3,0}-\dim E^3_{0,2}$. So,
\[
\dim H_2(\Wi N;\R)\ge \dim E^3_{0,2}\ge \dim E^3_{3,0} - \dim H_3(N;\R) = {2g \choose 3}-2g,
\]
where we use the facts that $E^{3}_{0,2}\cong H_2(\wt N;\R)_G$ and $E^3_{3,0}\cong H_3(G;\R)$, along with Propositions~\ref{prop: dim h_3(N)} and~\ref{prop: dim h_3(G)}. 
%note that because we are working in this proof in $\R$, it would suffice to have $b_3(N)=2g$, which is always the case.
\end{proof}

Since ${2g \choose 3}-2g>0$ if and only if $g\ge 3$, Proposition~\ref{prop: dim h_2(wt N)} implies the following.
\begin{cor}\label{nonzero}
We have $\pi_2(\wt N)\ne 0$ for $g\ge 3$.
\end{cor}

\section{Proof of Theorem~\ref{thm: main}}\label{sec: proof section}

In this section, we first study the (co)homology of the closed $4$-manifold $S_g$ and then, using results thus far, complete the proof of our main result.

\subsection{(Co)homology groups of $S_g$}\label{sec: cohomology Xg}
Recall from Section~\ref{subsec: construction} that $S_g$ is obtained by gluing two homeomorphic copies of the $\Sigma$-sphere complement $N$ along its smooth boundary $L$. For $g\ne 1$, the self-intersection $1-g$ of $\Sigma$ in $SP^2(M_g)$ is non-zero (\emph{cf.} Lemma~\ref{lem: pd u square}), so the bundle~\eqref{eq: bundle} is non-trivial and $L$ is a closed 3-dimensional lens space (which is homeomorphic to $S^3$ if and only if $g\in\{0,2\}$).

\begin{lemma}\label{lem: betti1}
For $g\ne 1$, we have that $H_1(S_g)\cong H_3(S_g)\cong \Z^{4g}$. 
\end{lemma}

\begin{proof}
Note that $H_4(N)\cong 0$ and $H_2(L)\cong 0$ since $g\ne 1$. 
    So, from the homology Mayer--Vietoris sequence for the pair $(N,L)$, we have 
    \[
   0\to H_4(S_g)\to H_3(L)\to H_3(N)\oplus H_3(N)\xlongrightarrow{j} H_3(S_g)\to 0.
    \]
Proceeding as in Proposition~\ref{prop: dim h_3(N)}, we get that $H_3(N)\cong \Z^{2g}$ for $g\ne 1$. Thus,
\[
\rank H_3(S_g) - 4g = \rank H_4(S_g)- \rank H_3(L)=1-1=0.
\]
%see, for instance,~\cite[Exercise~5.5]{Rot}. 
In particular, $b_3(S_g)=4g$. Therefore, $H_3(S_g)\cong\Z^{4g}$ in view of Poincar\'e duality and the fact that the first integral cohomology group is torsion-free. The universal coefficient theorem then gives $H_1(S_g)\cong \Z^{4g}$.
\end{proof}

\begin{prop}\label{prop: betti2}
    For $g\ge 2$, the closed smooth oriented $4$-manifold $S_g$ admits a near-symplectic structure.
\end{prop}

\begin{proof}
We first find $b_2(S_g)$. Note that by Lemma~\ref{lem: betti1} and~\eqref{eq: euler xg}, we get
    \[
    4g^2-10g+2=\chi(S_g)=2-8g+b_2(S_g),
    \]
which means $b_2(S_g)=4g^2-2g$ for $g\ne 1$. Since $H_1(S_g)\cong H^3(S_g)$ is torsion-free, the universal coefficient theorem implies that $H^2(S_g)\cong \Z^{4g^2-2g}\cong H_2(S_g)$. Next, because $\sigma(S_g)=0$ by~\eqref{eq: sign xg}, we also have 
\[
b_2^+(S_g)=2g^2-g=b_2^-(S_g)
\]
for $g\ne 1$. Since $b_2^+(S_g)\ge 1$ for $g\ge 2$, there exists a cohomology class whose square induces the given orientation of $S_g$, so it follows from~\cite{Taubes2,Honda} that there exists a near-symplectic form on $S_g$. More precisely, such near-symplectic forms are obtained for \emph{generic} Riemannian metrics on $S_g$ from self-dual harmonic forms (see also~\cite[Proposition~1]{ADK}).
\end{proof}

We deduce from our results that, similar to the case of $SP^2(M_g)$ for all $g\ge 0$, the (co)homology of $S_g=SP^2(M_g)\ \#_{\Sigma}\ \ov{SP^2(M_g)}$ is torsion-free for $g\ne 1$. 
%For cohomology, it is direct. For homology, it is by universal coefficient theorem (UCT). Indeed, if $H_2(S_g)$ has torsion, then so does $H^3(S_g)$ by UCT.
Moreover, it is also easy to show that
%\[ b_i(S_g)= 2b_i(N)\ \text{ for }i\in\{1,2,3\}. \]
%Thus, $b_1(N)=2g=b_3(N)$ and $b_2(N)=2g^2-g$, and 
the Euler characteristic of the compact $4$-manifold $N\cong SP^2(M_g)\setminus\Sigma$ for $g\ne 1$ is $\chi(N)=2g^2-5g+1$.

We conclude this subsection with the observation that the closed $4$-manifolds $S_g$ in Theorem~\ref{thm: main} are never complex surfaces (compare with Proposition~\ref{prop: no almost cplx}).

\begin{prop}\label{prop: no complex}
    For $g\ge 3$, the closed smooth oriented $4$-manifold $S_g$ cannot support a complex structure.
\end{prop}

\begin{proof}
In view of the construction of $S_g$, the Seifert--van Kampen theorem implies that $\pi_1(S_g)$ is the non-trivial amalgamated free product
\begin{equation}\label{eq: amalgam}
\pi_1(S_g)\cong \pi_1(N) \ast_{\Z_m} \pi_1(N).
\end{equation}
Thus, $\pi_1(S_g)$ cannot be the fundamental group of any closed K\"ahler manifold, see~\cite{Gr,ABR}. In particular, $S_g$ is not K\"ahler. But $b_1(S_g)$ is even by Lemma~\ref{lem: betti1}, so our assertion follows from the positive solution to Kodaira's conjecture for compact complex surfaces (see, for instance,~\cite[Remark~1.30]{ABCKT}).
\end{proof}

\subsection{Real Hurewicz map and second homotopy group}\label{sec: real Hurewicz}
Let $p_N\colon \wt N\to N$ and $p_g\colon\wt S_g\to S_g$ denote the universal coverings, where $g\ge 3$ and $m=g-1\ge 2$, and let $\Wi L=p^{-1}_N(L)$. Note that $\wt L=\bigsqcup S^3$ since the map $\Z_m\cong\pi_1(L)\to\pi_1(N)$ induced by $L\hookrightarrow N$ is injective (\emph{cf}. Proposition~\ref{prop: injective}). In view of~\eqref{eq: amalgam}, the inclusion $N\hookrightarrow S_g$ is $\pi_1$-injective and there is an embedding $\wt N\to\wt S_g$.

\begin{prop}\label{prop: pi_2 injective}
The map $\pi_2(\Wi N)\to\pi_2(\Wi S_g)$ induced by an embedding $\wt N\to \wt S_g$ is injective for $g\ge 3$. In particular, $\pi_2(S_g)\otimes\R\ne 0$ for $g\ge 3$.
\end{prop}
\begin{proof}
For the first assertion, it suffices to show, in view of the Hurewicz theorem, that the map $H_2(\Wi N)\to H_2(\Wi S_g)$ is injective. Let $Y$ be the closure of $\Wi S_g\setminus \Wi N$, so that $\Wi N\cap Y=\Wi L$. From the homology Mayer--Vietoris sequence for the pair $(\wt N,Y)$,
\[
\cdots\to H_2(\Wi L) \to H_2(\Wi N)\oplus H_2(Y)\to H_2(\Wi S_g)\to\cdots,
\]
and the fact that $H_2(\Wi L)=0$, it follows that $H_2(\Wi N)\to H_2(\Wi S_g)$ is injective. Now, Corollary~\ref{nonzero} implies that $\pi_2(S_g)\ne 0$ for $g\ge 3$. Since the second homotopy group of a closed $4$-manifold is always torsion-free (see Lemma~\ref{lem: pi2 torsion-free} below), it follows that $\pi_2(S_g)\otimes \R\cong\pi_2(S_g)\ne 0$.    
\end{proof}

\begin{prop}
    For $g\ge 3$, the second real Hurewicz map 
    \[
    h_2\otimes 1_\R\colon \pi_2(S_g)\otimes\R\to H_2(S_g;\R)
    \]
    is trivial. In particular, each closed differential $2$-form on $S_g$ is aspherical if $g\ge 3$.
\end{prop}

\begin{proof}
Let $\wh S_g$ be the normal covering of $S_g$ corresponding to the normal subgroup $\Z_m\subset \pi_1(N)\subset\pi_1(S_g)$. Then $\wh S_g$ is obtained from building blocks $\wh N$ (described in Subsection~\ref{subsec: homology/homotopy}) which are glued along copies of $L$. We can write 
\[
\wh S_g=\bigcup_{i=1}^\infty\wh N.
\]
Let $\wh S_g^n=\bigcup_{i=1}^n\wh N$ for $n\ge 1$. Note that $H_i(L;\R)=0$ for $i\in\{1,2\}$. Proceeding by induction on $n$ (the case $n=1$ being trivial in view of Proposition~\ref{prop: about H_2}), we use Mayer--Vietoris sequence for $\wh S_g^n$ to conclude that $H_2(\wh S_g^n;\R)=0$ for each $n\ge 1$. Hence, we have $H_2(\wh S_g;\R)=0$. It is then clear from the following commutative diagram that the real Hurewicz map $h_2\otimes 1\colon \pi_2(S_g)\otimes \R\to H_2(S_g;\R)$ is trivial:
\[
\begin{tikzcd}
    \pi_2(\wt S_g)\otimes \R\arrow{r}\arrow{d}{\cong} 
    &
    H_2(\wt S_g;\R) \arrow{d}
    \\
    \pi_2(\wh S_g)\otimes \R\arrow{r}\arrow{d}{\cong} 
    &
    H_2(\wh S_g;\R)=0 \arrow{d}
    \\
    \pi_2(S_g)\otimes \R\arrow{r}{h_2\otimes 1} 
    &
    H_2(S_g;\R).
\end{tikzcd}
\]
\end{proof}

This completes the proof of Theorem~\ref{thm: main}. Next, we prove Corollary~\ref{cor: each even dim}.

\begin{proof}[Proof of Corollary~\ref{cor: each even dim}]
    Suppose $g\ge 3$ and $n\ge 1$, and take $A_{g,n}=S_g\times T^{2n}$. Let $\omega_1$ be a near-symplectic form on $S_g$ (which exists --- see Proposition~\ref{prop: betti2}), and let $\omega_2$ be a symplectic form on $T^{2n}$. If $\pr_1\colon A_{g,n}\to S_g$ and $\pr_2\colon A_{g,n}\to T^{2n}$ are the factor projections, then the closed $2$-form $\omega:=\pr_1^*\omega_1+\pr_2^*\omega_2$ is a near-symplectic form on $A_{g,n}$ by definition (\emph{cf.} Subsection~\ref{subsec: near-symplectic}). 
    
Next, let $h_2^1$ be the real Hurewicz map for $S_g$, $h_2^2$ for $T^{2n}$, and $h_2$ for $A_{g,n}$. Then we have the following commutative diagram:
\[
\begin{tikzcd}[column sep=large]
    \left(\pi_2(S_g)\oplus \pi_2(T^{2n})\right)\otimes \R \arrow{r}{h_2^1\oplus h_2^2} \arrow{d}{\cong}
    &
    H_2(S_g;\R)\oplus H_2(T^{2n};\R)\arrow[d,hook]
    \\
    \pi_2(A_{g,n})\otimes \R\arrow{r}{h_2}
    &
    H_2(A_{g,n};\R).
\end{tikzcd}
\]    
It follows immediately that $\pi_2(A_{g,n})\otimes\R\ne 0$ and the map $h_2$ is trivial. 
%Hence, $A_{g,n}$ is a closed smooth totally $c$-symplectically aspherical $(2n+4)$-manifold with non-trivial second homotopy group. 

To see that $A_{g,n}$ is never K\"ahler, recall from~\eqref{eq: amalgam} that the finitely presented group $\pi_1(S_g)$ is a non-trivial amalgamated free product over $\Z_m$. So, Stallings's theorem on ends of groups~\cite{St} implies that $\pi_1(S_g)$ has infinitely many ends. Since $\pi_1(A_{g,n})\cong\pi_1(S_g)\oplus\Z^{2n}$ is an extension of $\pi_1(S_g)$ by the finitely generated group $\Z^{2n}$, it follows from~\cite{Gr,ABR} (see also~\cite{ABCKT}) that $A_{g,n}$ is not K\"ahler a manifold. This completes the proof. 
\end{proof}

For the sake of completeness, we record the following standard fact used in the proof of Proposition~\ref{prop: pi_2 injective}, which may be known to experts.

\begin{lemma}\label{lem: pi2 torsion-free}
    If $M$ is a (not necessarily compact) $4$-manifold without boundary, then $\pi_2(M)$ is torsion-free.
\end{lemma}

\begin{proof}
    It suffices to prove that $H_2(\wt M)\cong\pi_2(\wt M)$ is torsion-free. By Poincar\'e duality, $H_2(\wt M)\cong H^2_c(\wt M)$. If $H^2_c(\wt M)$ has torsion, then by the definition of cohomology with compact support, 
    %that \[ H^2_c(\wt M)\cong\lim_{\longrightarrow} H^2(\wt M,\wt M\setminus K), \]
    %where $K\subset \wt M$ are compact subsets. If $H^2_c(\wt M)$ has torsion, 
    there exists $K_0\subset \wt M$ such that $H^2(\wt M,\wt M\setminus K_0)$ has torsion. Then for the pair $(\wt M,\wt M\setminus K_0)$, we have the cohomology long exact sequence
    \[
    \cdots\to H^1(\wt M)\xlongrightarrow{i} H^1(\wt M\setminus K_0)\xlongrightarrow{\theta} H^2(\wt M,\wt M\setminus K_0)\xlongrightarrow{j} H^2(\wt M)\to \cdots,
    \]
    where $H^1(\wt M)\cong 0$, and the groups $H^1(\wt M\setminus K_0)\cong\text{Hom}(H_1(\wt M\setminus K_0),\Z)$ and $H^2(\wt M)\cong\text{Hom}(H_2(\wt M),\Z)$ are torsion-free. Let $0\ne x\in H^2(\wt M,\wt M\setminus K_0)$ be such that $nx=0$ for some $n\ge 2$. Then $j(x)=0$ and so, there exists $0\ne y\in H^1(\wt M\setminus K_0)$ such that $\theta(y)=x$. In particular, $\theta(ny)=0$. But $i$ is the trivial map, so $\theta$ is injective, which implies $ny=0$. This is a contradiction.
\end{proof}

%A similar argument shows that if $M$ is an $n$-manifold for $n\ge 5$ such that $\pi_i(M)\cong 0$ for $2\le i \le n-3$, then $\pi_{n-2}(M)$ is torsion-free.

\section{Sharpness of our results}\label{sec: exceptional}

There are no non-trivial $c$-symplectically aspherical examples with non-trivial second homotopy group in real dimension $2$. Indeed, it is easy to show that an aspherical $c$-symplectic form on a closed $2$-manifold $M$ exists if and only if $M$ is a closed Riemann surface of genus $\ge 1$, in which case $\pi_2(M)=0$. Hence, real dimension $4$ is the smallest possible dimension for the examples constructed in this paper. Moreover, in dimension $4$, our examples are sharp also in terms of the genus, as we explain next.

It is clear from the discussions in Sections~\ref{sec: homotopy of N} and~\ref{sec: proof section} that our techniques do not work for the $4$-manifolds $S_g=SP^2(M_g)\ \#_{\Sigma}\ \ov{SP^2(M_g)}$ in the cases $g\le 2$. In this section, we explain that, in fact, these three manifolds do not enjoy the interesting properties mentioned in Theorem~\ref{thm: main} for their higher genus counterparts.

\begin{prop}\label{prop: trivial pi2}
 For $g\in\{0,2\}$, we have $\pi_2(S_g)=0$, and for $g\in\{0,1\}$, the closed $4$-manifold $S_g$ cannot support an aspherical $c$-symplectic structure.  
\end{prop}

\begin{proof}
In the case $g=0$, we have $SP^2(M_0)\cong\C P^2$. So, the compact $4$-manifold $N\cong\C P^2\setminus\Sigma$ is homotopy equivalent to a $4$-ball with boundary $S^3$, and hence, homeomorphic to a $4$-ball due to Freedman~\cite{FQ}. Here, $S^3$ is the total space of the bundle~\eqref{eq: bundle} since the self-intersection of $\Sigma$ in $\C P^2$ is $1$ by Lemma~\ref{lem: pd u square}. Thus, $S_0$ is homeomorphic to the connected sum $S^4$. In particular, $\pi_2(S_0)\cong 0$. 

Next, in the case $g=2$, we have $SP^2(M_2)\cong T^4\#\ov{\C P^2}$, which is the K\"ahler blow-up of $T^4$ at a point. Then $N$ is homeomorphic to the once-punctured $4$-torus $\mathring{T^4}$ with boundary $L\cong S^3$, where the latter is the total space of the bundle~\eqref{eq: bundle} since the self-intersection of $\Sigma$ in $SP^2(M_2)$ is $-1$. Thus, $S_2\cong T^4\# T^4$. It is clear that 
\[
    \pi_2(S_2)\cong \pi_2(B(\Z^4\ast\Z^4)^{(3)})\cong\pi_2(B(\Z^4\ast\Z^4))\cong 0,
    \]
where $B(\Z^4\ast\Z^4)^{(3)}$ is the $3$-skeleton of the classifying CW complex $B(\Z^4\ast\Z^4)$ of $\pi_1(T^4\# T^4)\cong\Z^4\ast\Z^4$.

Finally, for the case $g=1$, observe that Proposition~\ref{prop: central extension} gives a central extension
\[
1\to\Z\rightarrow\pi_1(N)\rightarrow\Z^2\to 1
\]
due to which $\cd_{\Z}(\pi_1(N))= 3$.
%see, for instance,~\cite[Theorem~5.5]{Bieri}. 
Since $\pi_1(S_1)=\pi_1(N) \ast_{\Z} \pi_1(N)$ by~\eqref{eq: amalgam}, we get $\cd_{\R}(\pi_1(S_1))\le 3$. But, by~\eqref{eq: cd}, the fundamental group of a closed $c$-symplectically aspherical $4$-manifold $Y$ satisfies $\cd_{\R}(\pi_1(Y))\ge 4$. Hence, neither $S_1$ nor $S_0$ can support an aspherical $c$-symplectic structure.
\end{proof}

We conclude with the following observation.

\begin{prop}\label{prop: g=1}
    The closed oriented $4$-manifold $S_1$ is symplectic, but neither K\"ahler nor symplectically aspherical.
\end{prop}

\begin{proof}
When $g=1$, the Abel--Jacobi map $\mu_2\colon SP^2(M_1)\to T^2$ is an $S^2$-bundle (\emph{cf}. Subsection~\ref{sec: SP} and~\cite{OS}). The self-intersection of $\Sigma$ in $SP^2(M_1)$ is $0$, so the Euler class of the normal bundle to $\Sigma$ is zero and hence, the corresponding $S^1$-bundle~\eqref{eq: bundle} is trivial. Thus, $L\cong S^1\times S^2$ is the boundary of $N\cong SP^2(M_1)\setminus\Sigma$. In view of the bundle structure, removing the $2$-sphere $\Sigma$ from $SP^2(M_1)$ is parallel to removing a small $2$-ball $B^2$ from $T^2$, where $\pa(B^2\times S^2)\cong L\cong \pa N$. 

Then gluing $N$ and $\ov{N}$ along $L$ to form $S_1$ is parallel to gluing two copies of the once-punctured $2$-torus $\mathring{T^2}$ along $\pa(B^2)\cong S^1$ to form the closed Riemann surface $M_2$ of genus $2$. Thus, there is a smooth locally trivial fiber bundle 
\[
S^2\hookrightarrow S_1\to M_2,
\]
where $[S^2]\ne 0$ in $H_2(S_1;\R)$. As mentioned in Remark~\ref{rem: dragichi result}, Thurston's technique~\cite{Thu} (see also~\cite[Page~85]{Dra}) produces a symplectic form on $S_1$. That $S_1$ is not K\"ahler follows along the lines of the proof of Proposition~\ref{prop: no complex} in view of the fact that $\pi_1(S_1)\cong\pi_1(N)\ast_{\Z}\pi_1(N)$, and that it is not symplectically aspherical follows from Proposition~\ref{prop: trivial pi2}
\end{proof}

\end{document}